\documentclass[a4paper,11pt,psamsfonts]{amsart}
\usepackage{amsmath}
\usepackage{amsthm}
\usepackage{amssymb}
\usepackage{amscd}
\usepackage{amsfonts}
\usepackage{amsbsy}
\usepackage{epsfig,afterpage}
\usepackage{psfrag}
\usepackage{graphicx}
\usepackage{indentfirst, latexsym, bm,amssymb}
\usepackage{bbding}
\usepackage{enumerate}
\usepackage{xcolor}
\usepackage{setspace}
\usepackage{array}
\usepackage{color}

\newtheorem {theorem} {Theorem}%[section]
\newtheorem {proposition} [theorem]{Proposition}
\newtheorem {corollary} [theorem]{Corollary}
\newtheorem {lemma}  [theorem]{Lemma}

\newtheorem {remark} [theorem]{Remark}

\newcommand{\R}{\mathbb{R}}
\newcommand{\C}{\mathbb{C}}

\newcommand{\Q}{\mathbb{Q}}
\newcommand{\N}{\mathbb{N}}
\newcommand{\Z}{\mathbb{Z}}

\advance\textwidth by +1in
\advance\textheight by +1in
\advance\oddsidemargin by -0.5in
\advance\evensidemargin by -0.5in
\advance\topmargin by -0.5in

\begin{document}

\title[Meromorphic first integrals of diffeomorphisms]
{Meromorphic first integrals of analytic diffeomorphisms}

\author[A. Ferragut, A. Gasull,  X. Zhang]
{Antoni Ferragut$^1$, Armengol Gasull$^2$,  Xiang Zhang$^3$}

\address{$^1$ Universidad Internacional de la Rioja, Avenida de la Paz 137, 26006 Logro\~{n}o, Spain}
\email{toni.ferragut@unir.net}

\address{$^2$
Departament de Matem\`{a}tiques, Universitat Aut\`{o}\-no\-ma de Barcelona
and Centre de Recerca Matem\`{a}tica, 08193 Bellaterra, Barcelona,
Spain} \email{gasull@mat.uab.cat}

\address{$^3$ School of Mathematical Sciences,  and MOE--LSC, Shanghai Jiao Tong
University, Shanghai, 200240, P. R. China}
\email{xzhang@sjtu.edu.cn}



\date{}

\begin{abstract}
We give an upper bound for the number of functionally independent
meromorphic first integrals that a discrete dynamical system
generated by an analytic map $f$ can have in a neighborhood of one
of its fixed points. This bound is obtained in terms of the resonances
 among the eigenvalues of the differential of $f$ at this  point. Our
 approach is inspired on similar Poincar\'{e} type results for ordinary
differential equations. We also apply our results to several
examples, some of them motivated by the study of several difference
equations.
\end{abstract}

\maketitle

\section{Introduction and statement of the main results}%\label{s1}

One of the first steps to study the dynamics of a discrete dynamical
system (DDS) is to determine the number $m$ of functionally
independent  first integrals that it has.  It is clear that each new
first integral reduces the region where any orbit can lie, so the
bigger is $m,$ the simpler will be the dynamics. For instance, if
this DDS is $n$-dimensional, this number $m$ is at most $n,$ and in
the case $m=n$ the DDS is called integrable and it has extremely simple dynamics: in most cases
it is globally periodic, that is, there exists  $p\in\N,$ such that
$f^p=\operatorname{Id},$ where $f$ is the invertible map that
generates it, see \cite{CGM}. Similarly, DDS having $m=n-1$ are such
that all their orbits lie in one-dimensional manifolds, see some
examples in \cite{CGM2}.

The aim of this paper is to give an upper bound of the number of
meromorphic first integrals that a DDS generated by an invertible
analytic map can have in a neighborhood of a fixed point. We follow
the approach of Poincar\'{e} for studying the same problem for
continuous dynamical systems given by analytic ordinary differential
equations. It is based on the study of the resonances among the
eigenvalues of the differential of the vector field at one of its
critical points, see for instance \cite{CLZ11} and their references.
We will use similar tools that the ones introduced in that paper.

\smallskip

Consider analytic diffeomorphisms in $(\mathbb C^n,0),$ a
neighborhood of the origin,
\begin{equation}\label{e1.1}
y=f(x),\qquad x\in (\mathbb C^n,0),
\end{equation}
with $f(0)=0$. A function $R(x)=G(x)/H(x)$ with $G$ and $H$ analytic
functions in $(\mathbb C^n,0)$ is a {\it meromorphic first integral}
of the diffeomorphism \eqref{e1.1} if
\[
G(f(x))H(x)=G(x)H(f(x)), \qquad \mbox{ for all }  x\in (\mathbb
C^n,0).
\]
Notice that the above condition implies that
\[
R(f(x))=R(x),
\]
for all $x\in (\mathbb C^n,0)$ where both functions are well
defined. Specially if $G$ and $H$ are polynomial functions, then
$R(x)$ is a {\it rational first integral} of \eqref{e1.1}. If $H$ is
a non--zero constant, then $R(x)$ is  an {\it analytic first
integral} of \eqref{e1.1}. So meromorphic first integrals include
rational and analytic first integrals as particular cases.

\smallskip

Denote by $A={\rm D} f (0)$ the Jacobian matrix of $f(x)$ at $x=0$.
Let $\mu=(\mu_1,\ldots,\mu_n)$ be the $n$--tuple of eigenvalues of
$A$. Notice that since $f$ is a diffeomorphism at $0,$ we have
$\mu_1\mu_2\cdots\mu_n\ne0.$

We say that the eigenvalues $\mu$ satisfy a {\it resonant condition}
if
\[
\mu^{\mathbf k}=1, \qquad\mbox{ for some } \mathbf  k\in \mathbb
Z^n\quad \mbox{with}\quad \|\mathbf k\|\ne0,
\]
where $\mathbb Z$ is the set of integers, and $\|\mathbf
k\|=|k_1|+\dots+|k_n|$.

\smallskip

The aim of this paper is to prove the next result and give some
applications.

\smallskip

\begin{theorem}\label{t1}
Assume that the analytic diffeomorphism \eqref{e1.1} satisfies
$f(0)=0$ and let $\mu=(\mu_1,\ldots,\mu_n)$ be the eigenvalues of
${\rm D} f (0)$. Then the number of functionally independent
generalized rational first integrals of the analytic diffeomorphism
\eqref{e1.1} in $(\mathbb C^n,0)$ is at most the dimension of the
$\mathbb Z$--linear space generated from $\{\mathbf k\in\mathbb
Z^n:\, \mu^{\mathbf k}=1\}$.
\end{theorem}

\smallskip

The above theorem also extends  some results proved for $n=2$ in
\cite{CGM3}.

\smallskip

In section \ref{s3} we prove Theorem \ref{t1}. Section \ref{sapp} is
devoted to applications of Theorem \ref{t1}.

\section{Proof of Theorem \ref{t1}} \label{s3}

For an analytic or a polynomial function $R(x)$ in $(\mathbb
C^n,0)$, we denote by $R^0(x)$ its homogeneous term of the lowest
degree. For a rational or a meromorphic function $R(x)=G(x)/H(x)$ in
$(\mathbb C^n,0)$, we denote by $R^0(x)$ the rational function
$G^0(x)/H^0(x)$. We expand the analytic functions $G(x)$ and $H(x)$
as
\[
G^0(x)+\sum\limits_{i=1}\limits^{\infty}G^i(x)\quad \mbox{ and
}\quad H^0(x)+\sum\limits_{i=1}\limits^{\infty}H^i(x),
\]
where $G^i(x)$ and $H^i(x)$ are homogeneous polynomials of degrees
$\deg{G^0(x)}+i$ and $\deg{H^0(x)}+i$, respectively. Then we have
\begin{eqnarray}\label{e3.5}
R(x)=\frac{G(x)}{H(x)}&=&\left(\frac{G^0(x)}{H^0(x)}+\sum\limits_{i=1}
\limits^{\infty}\frac{G^i(x)}{H^0(x)}\right)\left(1+\sum
\limits_{i=1}
\limits^{\infty}\frac{H^i(x)}{H^0(x)}\right)^{-1}\nonumber\\
&= &
\frac{G^0(x)}{H^0(x)}+\sum\limits_{i=1}\limits^{\infty}\frac{A^i(x)}{B^i(x)},
\end{eqnarray}
where $A^i(x)$ and $B^i(x)$ are homogeneous polynomials. Clearly,
\[
\deg{G^0(x)}-\deg{H^0(x)}<\deg {A^i(x)}-\deg{B^i(x)} \qquad \mbox{
for all } i\ge 1.
\]
In what follows we will say that $\deg {A^i(x)}-\deg{B^i(x)}$ is the
degree of $A^i(x)/B^i(x)$, and $G^0(x)/H^0(x)$ is the lowest degree
term of $R(x)$ in the expansion \eqref{e3.5}.  For simplicity we
denote
\[
d(G)=\deg{G^0(x)}, \qquad d(R)=d(G)-d(H)=\deg{G^0(x)}-\deg{H^0(x)},
\]
and call $d(R)$ the {\it lowest degree} of $R$.

\smallskip

The following known result, first proved by Ziglin \cite{zig} in
1983, see also \cite{bai,CLZ11,ito,Zh17}, will be used in the proof
of Theorem \ref{t1}.

\smallskip

\begin{lemma}\label{l3.2} Let
\[
R_1(x)=\frac{G_1(x)}{H_1(x)},\ldots, R_m(x)=\frac{G_m(x)}{H_m(x)},
\]
be functionally independent meromorphic functions in $(\mathbb
C^n,0)$. Then there exist polynomials $P_i(z_1,\ldots,z_i)$ for
$i=2,\ldots, m$ such that $R_1(x),$ $\widetilde
R_2(x)=P_2(R_1(x),R_2(x)),$ $\ldots, \widetilde
R_m(x)=P_m(R_1(x),\ldots,R_m(x))$ are functionally independent
meromorphic functions, and that $R_1^0(x),$ $\widetilde R_2^0(x),$
$\ldots, \widetilde R_m^0(x)$ are functionally independent rational
functions.
\end{lemma}

\smallskip

Next we give some properties that meromorphic first integrals of
diffeomorphism \eqref{e1.1} must have. A {\it rational monomial} is
by definition the ratio of two monomials, i.e. of the form
$x^{\mathbf p}/x^{\mathbf q}$ with $\mathbf p,\mathbf q\in
(\mathbb{N}_0 )^n,$ where  $\mathbb{N}_0 =\mathbb N\cup\{0\}$ and
$\mathbb N$ is the set of positive integers. The rational monomial
$x^{\mathbf p}/x^{\mathbf q}$ is {\it resonant} if $\mu^{\mathbf
p-\mathbf q}=1$. A rational function is {\it homogeneous} if its
denominator and numerator are both homogeneous polynomials. A
rational homogeneous function is {\it resonant} if the ratio of any
two elements in the set of all its monomials in both denominator and
numerator is a resonant rational monomial.

\smallskip

For the diffeomorphism \eqref{e1.1} we assume without loss of
generality that $A={\rm D} f (0)$ is in its Jordan normal form and
is a lower triangular matrix.

\smallskip

\begin{lemma}\label{l3.3}
Let  $R(x)=G(x)/H(x)$ be a meromorphic first integral of the
analytic diffeomorphism \eqref{e1.1}. By changing $R$ by $R-a,$ for
some suitable constant $a\in \mathbb C,$ if needed,  it is not
restrictive to assume that $R^0(x)=G^0(x)/H^0(x)$ is  non--constant.
Moreover $R^0$ is a resonant rational homogeneous first integral of
the linear part of $f(x).$
\end{lemma}

\smallskip

For proving this last lemma we will use the following result, for a
proof see any of the references \cite{bib, LLZ, Zh13,Zh17}.

\begin{lemma}\label{l3.4}
Let $\mathcal H_n^p$ be the linear space of complex coefficient
homogeneous polynomials of degree $p$ in $n$ variables. For any
constant $c\in\mathbb C$, define a linear operator on $\mathcal
H_n^p$ by
\[
\mathcal{L}_c(h)(x)=h(Ax)-c\,h(x), \qquad h(x)\in \mathcal H_n^p.
\]
Then the spectrum of $\mathcal{L}_c$ is
\[
\{\mu^{\mathbf k}-c:\,\,\mathbf k\in(\mathbb{N}_0 )^n,\, |\mathbf
k|=k_1+\ldots+k_n=p\},
\]
where $\mu$ are the eigenvalues of $A$.
\end{lemma}

\smallskip

\begin{proof}[Proof of Lemma \ref{l3.3}] If $R^0(x)=G^0(x)/H^0(x)\equiv
a$ is  constant, the function $R(x)-a$ is also a meromorphic first
integral and $(R(x)-a)^0$ is not constant. Hence, without loss of
generality, we can assume that $R$ is a meromorphic first integral
and that  $R^0$ is not identically constant.

Let us prove that $R^0$ is a resonant rational homogeneous first
integral. As in \eqref{e3.5} we write $R(x)$ as
\[
R(x)=R^0(x)+\sum\limits_{i=1}\limits^\infty R^i(x),
\]
where $R^0(x)$ is the lowest order rational homogeneous function and
$R^i(x)$ for $i\in \mathbb N$ are rational homogeneous functions of
order larger than $R^0(x)$. Since $R(x)$ is a first integral of the
diffeomorphism $f(x)$ in a neighborhood of $0\in \mathbb C^n$, we
have
\[
R(f(x))=R(x),\qquad x\in(\mathbb C^n,0).
\]
Equating the lowest order rational homogeneous functions we get
\begin{equation}\label{e3.6}
R^0(Ax)=R^0(x),\qquad i.e. \quad
\frac{G^0(Ax)}{H^0(Ax)}=\frac{G^0(x)}{H^0(x)}.
\end{equation}
This implies that $R^0(x)$ is a rational homogeneous first integral
of the linear part of the analytic diffeomorphism \eqref{e1.1}.

\smallskip

Next we shall prove that $R^0(x)$ is resonant. {From} the equality
\eqref{e3.6} we can assume without loss of generality that $G^0(x)$
and $H^0(x)$ are relative prime. Now equation \eqref{e3.6} can be
written as
\[
H^0(x)G^0(Ax)=G^0(x)H^0(Ax).
\]
Since $G^0$ and $H^0$ are relatively prime, and $\mathbb C[x]$ is a
unique factorization domain (see e.g. \cite[p. 2]{Fu69}), we get
that $G^0(x)$ divides $G^0(A x)$ and $H^0(x)$ divides $H^0(A x)$. In
addition, $G^0(Ax)$ and $G^0(x)$ have the same degree, so there
exists a constant $c$ such that
\[
G^0(Ax)-c\,G^0(x)\equiv 0,\quad H^0(Ax)-c\,H^0(x)\equiv 0.
\]
We remark that if $G^0(x)\equiv 1$ or $H^0(x)\equiv 1$, we have
$c=1$. Set $\deg G^0(x)=l$, $\deg H^0(x)=m$ and $\mathcal{L}_c$
the linear operator defined in Lemma \ref{l3.4}. Recall from Lemma
\ref{l3.4} that $\mathcal{L}_c$ has respectively the spectrums on
$\mathcal H_n^p$
\[
\mathcal S_p:=\{\mu^{\textbf{p}}-c:\,\,\textbf{p}\in(\mathbb
Z^+)^n,\, |\textbf{p}|=p\},
\]
and on $\mathcal H_n^q$
\[
\mathcal S_q:=\{\mu^{\textbf{q}}-c:\,\,\textbf{q}\in(\mathbb
Z^+)^n,\, |\textbf{q}|=q\}.
\]

\smallskip

Separate $\mathcal H_n^p=\mathcal H_{n,1}^p+\mathcal H_{n,2}^p$ in
such a way that for any $p(x)\in \mathcal H_{n,1}^l $ its monomial
$x^{\mathbf p}$ satisfies $\mu^{\textbf{p}}-c=0$, and for any
$Q(x)\in \mathcal H_{n,2}^p $ its monomial $x^{\mathbf p}$ satisfies
$\mu^{\textbf{p}}-c\ne 0$. Separate $G^0(x)$ in two parts
$G^0(x)=G_1^0(x)+G_2^0(x)$ with $G_1^0\in\mathcal H_{n,1}^p$ and
$G_2^0\in\mathcal H_{n,2}^p$. Since $A$ is in its Jordan normal form
and is lower triangular, it follows that
\[
\mathcal{L}_c\mathcal H_{n,1}^p\subset \mathcal H_{n,1}^p, \quad
\mbox{and} \quad \mathcal{L}_c\mathcal H_{n,2}^p\subset \mathcal
H_{n,2}^p.
\]
Hence $\mathcal{L}_c G^0(x)\equiv 0$ is equivalent to
\[
\mathcal{L}_c G_1^0(x)\equiv 0\quad \mbox{and } \quad \mathcal{L}_c
G_2^0(x)\equiv 0.
\]
Since $\mathcal{L}_c $ has the spectrum without zero element on
$\mathcal H_{n,2}^p$ and so it is invertible on $\mathcal
H_{n,2}^p$, the equation $\mathcal{L}_c G_2^0(x)\equiv 0$ has only
the trivial solution, i.e. $G_2^0(x)\equiv 0$. This proves that
$G^0(x)=G_1^0(x)$, i.e. each monomial, say  $x^{\mathbf p}$, of
$G^0(x)$ satisfies $\mu^{ {\mathbf p}}-c=0$.

\smallskip

Similarly we can prove that each monomial, say  $x^{\mathbf q}$, of
$H^0(x)$ satisfies $\mu^{{\mathbf q}}-c=0$. This implies that
$\mu^{{\mathbf p-\mathbf q}}=1$. The above proofs show that
$R^0(x)=G^0(x)/H^0(x)$ is a resonant rational homogeneous first
integral of the linear part of $f(x)$. \end{proof}

\smallskip

Having the above lemmas we can prove Theorem \ref{t1}.

\begin{proof}[Proof of Theorem \ref{t1}]
Let
\[
R_1(x)=\frac{G_1(x)}{H_1(x)},\,\ldots,\,
R_m(x)=\frac{G_m(x)}{H_m(x)},
\]
be the $m$ functionally independent meromorphic first integrals of
the diffeomorphism $f$. Since the polynomial functions of $R_i(x)$
for $i=1,\ldots,m$ are also meromorphic first integrals of the
diffeomorphsim $f$,  by Lemma \ref{l3.2} we can assume without
loss of generality that
\[
R_1^0(x)=\frac{G_1^0(x)}{H_1^0(x)},\ldots,R_m^0(x)=\frac{G_m^0(x)}{H_m^0(x)},
\]
are functionally independent.

\smallskip

Lemma \ref{l3.3} shows that $R_1^0(x),\ldots,R_m^0(x)$ are resonant
rational homogeneous first integrals of the linear part $Ax$ of
$f(x)$. So these first integrals can be written as rational
functions in the variables given by resonant rational monomials.
Write $A=A_S+A_N$ with $A_S=\mbox{diag}(\mu_1,\ldots,\mu_n)$ and
$A_N$ nilpotent. Then direct calculations show that any resonant
rational monomial is a first integral of $A_S x$. For instance, let
$x^{\bf k}$ be a resonant rational monomial, then $\mu^{\bf k}=1$.
Hence, we have $(A_S x)^{\bf k}=\mu^{\bf k}x^{\bf k}=x^{\bf k}$.
This implies that $R_1^0(x),\ldots,R_m^0(x)$ are also first
integrals of $A_S x$. We claim that $m$ is less than or equal to the
number of elements in a basis of the $\mathbb Z$--linear space
generated from $\{\mathbf k\in \mathbb Z^n:\, \mu^{\mathbf k}=1\}$.
We denote by $\gamma$ this last number.

Indeed, since $R_1^0(x),\ldots,R_m^0(x)$ are the first integrals of
$A_S x$, we only need to prove that the number of functionally
independent resonant rational homogeneous first integrals of $A_S x$
is equal to $\gamma$. In fact, from Lemma \ref{l3.3} and its proof
it follows that any resonant rational homogeneous first integral
consists of resonant rational monomials. This means that the maximum
number of functionally independent resonant rational homogeneous
first integrals is equal to the maximum number of functionally
independent resonant rational monomials. Whereas, by definition a
resonant rational monomial, saying $x^{\mathbf k}$ for $\mathbf k\in
\mathbb Z^n$, satisfies $\mu^{\mathbf k}=1$. This proves the claim.

\smallskip

After the claim, the proof of the theorem is completed. \end{proof}

\section{Applications}\label{sapp}

\subsection{Some simple examples} We start studying the integrability of two simple
examples coming from second order difference equations. Recall that
the study of the sequences generated by  second order difference
equations $x_{n+2}=g(x_n,x_{n+1})$ can be reduced to the study of
the DDS generated by the planar map $f(x,y)=(y,g(x,y)).$ Moreover,
the first integrals $H$ for the DDS are usually called {\it
invariants} for the difference equation. Then
$H(x_n,x_{n+1})=H(x_{n+1},x_{x+2})$ for all $n\in\N.$

As a first example we study the integrability of the planar map
\begin{equation}\label{eq:ex1}
f(x,y)=(y,-bx+c/y), \quad   b,c\in\C,b\ne0.
\end{equation}
coming from the difference equation $x_{n+2}=-b x_n +c/x_{n+1}.$  We
prove the following proposition.

\begin{proposition}\label{pr:1} If the planar map \eqref{eq:ex1} has a
meromorphic first integral then $b$ is a root of the unity. In
particular, if $b,c\in\R$ and the map has a meromorphic first
integral then $b\in\{-1,1\}.$
\end{proposition}

Somehow the result for the real case is sharp because when $b=1$ the
function $H(x,y)=x^2y^2-cxy$ is a first integral of \eqref{eq:ex1}.
Moreover when $b=-1$ and $c=0,$ clearly the map is linear and
integrable. The first integral for the case $b=1$ (in fact for its
inverse) and also for many other rational maps are given in
\cite{tor}.  For the complex case, observe that when $c=0$ and $b$
is a root of the unity the map is linear and globally periodic, i.e.
$f^n=\operatorname{Id}$ for some $n\in\N.$  For instance, by using
the results of \cite{CGM} it is easy to construct two functionally
independent rational first integrals for each one of them.

\begin{proof}[Proof of Proposition \ref{pr:1}]  When $c=0$
the map has $(0,0)$ as a fixed point and the eigenvalues of ${\rm D} f
(0,0)$ are $\pm \sqrt b.$ Then, by Theorem \ref{t1}, to have a
meromorphic first integral there must exist $(m,n)\in\Z^2\setminus
(0,0)$ such that $(\sqrt b)^n(-\sqrt b)^m=1$ and, as a consequence
$b$ must be a root of the unity.

\smallskip

From now on we can assume that
$c\ne0.$ Then, the fixed points of $f$ are $(z,z)$ where $(b+1)z^2=c.$
If $b=-1$ we are done because $b^2=1.$ Otherwise, the map has two
(real or complex) fixed points.  Anyhow, for $c(b+1)\ne0,$
\[
{\rm D} f (z,z)= \left(
           \begin{array}{cc}
             0 & 1 \\
             -b & -c/z^2 \\
           \end{array}
         \right)=
         \left(
           \begin{array}{cc}
             0 & 1 \\
             -b & -(b+1) \\
           \end{array}
         \right).
\]
Hence, its eigenvalues satisfy
$p(\mu )=\mu ^2+(b+1)\mu +b=0.$ Since $b\ne0$ we can write
them as $\mu$ and $b/\mu.$ By Theorem \ref{t1}, if $f$ is
meromorphically integrable, there exists $(m,n)\in\Z^2\setminus
(0,0)$ such that $\mu^n(b/\mu)^m=1.$ If $m=n,$ then $b^m=1$ and we
are done, again. Hence, from now on, $m-n\ne0.$ Moreover, $\mu=
b^{m/(m-n)}.$ Thus
\[
p(\mu)=\mu^2+(b+1)\mu+b=b^{\frac{2m}{m-n}}+(b+1)b^{\frac{m}{m-n}}+b=0.
\]
If we introduce $a$ such that $b=a^{m-n},$ the above equality writes
as
\[
a^{2m}+\big(a^{m-n}+1\big)a^m+a^{m-n}=a^{m-n}\big(a^{m+n}+a^m+a^n+1
\big)= a^{m-n}(a^m+1)(a^n+1)=0.
\]
Hence $a$ must be a root of the unity, and as a consequence
$b=a^{m-n}$ also must, as we wanted to prove.
\end{proof}

As a second example,  next lemma  studies conditions for a planar
map to have two functionally independent meromorphic first
integrals.

\begin{lemma}\label{le:n2} Let $f$ be a real analytic planar map with a fixed point
${\bf x}\in\R^2$ such that the characteristic polynomial of ${\rm D}
f ({\bf x})$ is $p(\mu )=\mu ^2+b\mu +c\in\R[\mu ].$
Assume that $f$ has two functionally independent meromorphic  first
integrals. Then, the eigenvalues of ${\rm D} f ({\bf x})$ are roots
of the unity. In particular, if $b^2-4c<0$ then $c=1$ and
otherwise $(b,c)\in\{ (\pm 2,1), (0,-1)\}.$
\end{lemma}

\begin{proof} Let $u,v\in\C$ be the two eigenvalues of ${\rm D} f ({\bf x})$.
By Theorem \ref{t1}, there exist $(n,m)$ and $(n',m')$ in $\Z^2$
such that
\[
u^nv^m=1,\quad u^{n'}v^{m'}=1,\quad \mbox{and}\quad mn'-m'n\ne0.
\]
Hence $u^{nm'-n'm}=1$ and $v^{mn'-m'n}=1.$ Therefore, $u$ and $v$
are roots of the unity. When $u$ and $v$ are complex we are done,
because $v=\bar u$ and $c=u\bar u=|u|^2=1.$ When $u,v\in\R$ then $u$
and $v$ are either $1$ or $-1$ and the lemma follows.
\end{proof}

We apply the above lemma to study the map
\begin{equation}\label{eq:ex2}
f(x,y)=(y,x^py^q), \quad   p,q\in\Z,\,\, p\ne0,
\end{equation}
that  describes the difference equation $x_{n+2}=x_n^p x_{n+1}^q.$
It has the fixed point $(1,1)$ for all values of $p$ and $q,$ and
provides a good test for our result because it can also be studied
by another approach. It can be linearized on $\{(x,y)\in\R^2\,:\,
x>0, y>0\},$ because with the new variables $(u,v)=(\ln x,\ln y)$
the DDS generated by $f$ is conjugated to the DDS generated by
$g(u,v)=(v, p u+ q v).$

As we have already commented, we will center our attention in finding
the values of $p$ and $q$ such that the DDS generated by $f$ can
have two functionally independent first integrals. Notice that
$f(1,1)=(1,1)$ and
\[
{\rm D} f (1,1)= \left(
           \begin{array}{cc}
             0 & 1 \\
             px^{p-1}y^q & qx^py^{q-1} \\
           \end{array}
         \right)_{(x,y)=(1,1)}=
         \left(
           \begin{array}{cc}
             0 & 1 \\
             p & q \\
           \end{array}
         \right).
\]
Hence the eigenvalues $\mu $ of ${\rm D} f (1,1)$ satisfy $\mu
^2-q\mu -p=0.$ We can use Lemma \ref{le:n2}: since $(p,q)\in\Z^2$,
the only cases for which the map \eqref{eq:ex2} can have two
functionally independent first integrals are
$(p,q)\in\{(-1,-2),(-1,2),(1,0)\}$ when $q^2+4p\ge0,$ and $p=-1$
when $q^2+4p<0.$ Moreover, in this last case all roots of $P(\mu
)=\mu ^2-q\mu +1$ must be roots of the unity.  This only happens
when $q\in\{-1,0,1\},$ because these values are the only ones for
which $P$ is a quadratic cyclotomic polynomial. Recall that the {\it
$p$-th cyclotomic polynomial} $\Phi_p$  is the monic polynomial with
integer coefficients, and irreducible in $\Q(x)$, such that its
roots are the primitive $p$-roots of the unity, that is ${\rm
e}^{2n\pi{\rm i}/p},$ for $p$ and $n$ relatively prime. Recall also
that the only quadratic cyclotomic polynomials correspond to 3$^{\rm
rd}$, 4$^{\rm th}$ and 6$^{\rm th}$ roots of the unity. Hence, from
our tools, we have seen that there are six maps of the form
\eqref{eq:ex2} that are candidates to have two functionally
independent first integrals:
\begin{align*}f_1(x,y)&=\Big(y,\frac1{xy^2}\Big),
&& f_2(x,y)=\Big(y,\frac{y^2}{x}\Big), && f_3(x,y)=\big(y,x\big),\\
f_4(x,y)&=\Big(y,\frac1{xy}\Big), &&
f_5(x,y)=\Big(y,\frac{1}{x}\Big), &&
 f_6(x,y)=\Big(y,\frac y x\Big).
\end{align*}

From these candidates it is not difficult to see, by using the tools
of \cite{CGM}, that $f_j(x,y)$ $j=3,4,5,6$ have the desired
property.   For instance, two functionally independent first
integrals for $f_6$ are
\[
H_1(x,y)=x+\frac1x+y+\frac1y+\frac x y+ \frac y x, \quad H_2(x,y)=
xy+\frac 1{xy}+\frac{x^2}{y}+\frac{y}
{x^2}+\frac{x}{y^2}+\frac{y^2}x.
\]
In fact, these four maps are globally periodic, with periods $2,3,4$
and $6,$ respectively.

\subsection{Planar rational  maps with rational coefficients}

We consider now rational maps
\begin{equation}\label{ratmap1}
f(x,y)=(R_1(x,y),{R_2(x,y)})=\left(\frac{P_1(x,y)}{Q_1(x,y)},\frac{P_2(x,y)}{Q_2(x,y)}\right),
\end{equation}
with coefficients in $\Q,$ that is, with $P_i,Q_i\in\Q[x,y]$ for
$i=1,2.$

\smallskip

We introduce some notation. As usual, when two integer numbers $p$
and $n$ are coprime we will write $(p,n)=1.$ Moreover we will denote
the $p$-th cyclotomic polynomial $\Phi_p$ by $\phi(p).$ In fact,
$\phi(p)$ is the number of integers between 1 and $p$ that are
relatively prime to $p.$

\smallskip

Given two polynomials $P,Q\in\C[x],$ as usual, we denote by
$\operatorname{Res}_x(P(x),Q(x))\in\C$  the resultant of $P$ and $Q$
with respect to  $x.$ Recall that $P$ and $Q$ share some complex
root if and only if $\operatorname{Res}_x(P(x),Q(x))=0,$ see
\cite{stu}.

\smallskip

We  will also need the following  result, which is essentially
contained in (\cite{BDeA,WZ}).

\begin{proposition}\label{pr:trig}   For each
$p\in\N$ there is a polynomial  $M_m\in\Q[x]$ of minimal degree
$m=\phi(p)/2$ such that $M_m(\cos (2n\pi/p))=0$ for all $(p,n)=1.$
Moreover,
\begin{equation}\label{eq:mm}
\operatorname{Res}_x\big(\Phi_p(x),x^2-2xv+1 \big)=M_m^2(v).
\end{equation}
In particular, $m=1$ if and only if $p\in\{1,2,3,4,6\}$; $m=2$ if
and only if $p\in\{5,8,10,12\}$; $m=3$ if and only if
$p\in\{7,9,14,18\}$; and $m=(p-1)/2$ when $p>2$ is prime.
\end{proposition}

We only make some comments about \eqref{eq:mm}. Notice that if
$x=\cos (2n\pi/p)+{\rm i}\sin (2n\pi/p)$ then $\Phi_p(x)=0.$
Moreover, if we define $v=\cos (2n\pi/p)$ it holds that $v=(x+\bar
x)/2= (x+1/x)/2,$ or equivalently $(x^2-2xv+1)/x=0.$ Hence the
polynomial $\operatorname{Res}_x\big(\Phi_p(x),x^2-2xv+1 \big)$ has
$v$ as one of its roots.  Moreover, if we start with $\bar x$
instead of $x$ we obtain again the same value $v$ as a root of this
last polynomial, making that all its roots are double.

\smallskip

Notice that, as a consequence of the above result, the only rational
values of $\cos(2\pi/p)$   are
\begin{equation}\label{eq:rat}
\cos(2\pi)=1,\, \cos(2\pi/2)=-1,\, \cos (2\pi/3)=-1/2,\,
\cos(2\pi/4)=0,\, \cos(2\pi/6)=1/2,\end{equation} and the only
values where $\cos(2\pi/p)$ is in $\Q[\sqrt{q}]$ for $q\in\Q,$ but
$\sqrt{q}\not\in\Q$ are
\begin{equation}\label{eq:sqr}
\cos(2\pi/5)=(\sqrt 5-1)/4,\quad \cos(2\pi/10)=(\sqrt 5+1)/4,\quad
\cos (2\pi/12)=\sqrt 3/2.
\end{equation} For instance, the minimal
polynomial for $p=5$ is $M_2(x)=4x^2+2x-1$; the cases where $M_m$
is cubic are, $8x^3+4x^2-4x-1$ when $p=7;$ $8x^3-6x+1$ when $p=9;$
$8x^3-4x^2-4x+1$ when $p=14;$ and $8x^3-6x-1$ when $p=18.$

\smallskip

Next result provides some computable conditions to know whether a
planar rational map with rational coefficients has a meromorphic
first integral. In particular, notice that for many 1-parametric
families of maps it allows to prove that this first integral
can only exist for finitely many values of the parameter, see for
instance Lemma \ref{le:ex1}.

\begin{theorem}\label{te:t0}
Consider the rational map \eqref{ratmap1} with rational
coefficients. Assume that $f$ has a real fixed point $(\widehat
x,\widehat y)\in\R^2$ and that ${\rm D} f (\widehat x,\widehat y)$ has
complex conjugated eigenvalues $\mu ,\bar\mu$ with modulus
different from 1. If the map \eqref{ratmap1} has a meromorphic
first integral, then $\mu /\bar\mu $ is a root of the unity
and there exists a computable polynomial $U_k\in\Q[x],$ of degree
$k\in\N,$ such that
$U_k\big(\operatorname{Re}(\mu /\bar\mu )\big)=0.$ Moreover,
some of the values \begin{equation}\label{eq:con}
\operatorname{Res}_x \big(U_k(x), V_{p}(x)\big),\ \mbox{with}\ p
\ \mbox{such that}\ \deg(V_p(x))\le k,
\end{equation}
must vanish, where $V_{p}$ is the minimal polynomial of
$\operatorname{Re}(\mu /\bar\mu )=\cos(2 n\pi/p),$ with
$(p,n)=1,$ see Proposition \ref{pr:trig}.
\end{theorem}

\begin{proof} We introduce the polynomials $S_1,S_2,T_1,T_2,D_1,D_2$
associated  to $f,$
\begin{align*}%\label{eq:not}
S_1(x,y)&= P_1(x,y)-xQ_1(x,y),\quad S_2(x,y)=P_2(x,y)-yQ_2(x,y),\nonumber\\
T(x,y)&=\frac{T_1(x,y)}{T_2(x,y)}=\frac{\partial R_1(x,y)}{\partial
x}+\frac{\partial R_2(x,y)}{\partial y},\\
D(x,y)&=\frac{D_1(x,y)}{D_2(x,y)}=\frac{\partial R_1(x,y)}{\partial
x}\frac{\partial R_2(x,y)}{\partial x}-\frac{\partial
R_1(x,y)}{\partial y}\frac{\partial R_2(x,y)}{\partial x}.\nonumber
\end{align*}

Notice that if $(\widehat x,\widehat y)$ is a fixed point of $f$ then
it is a solution of the system $\{S_1(x,y)=0,\, S_2(x,y)=0\},$ and
moreover $Q_1(\widehat x,\widehat y)Q_2(\widehat x,\widehat y)\ne0.$
Observe also that the eigenvalues $\mu =\mu (\widehat
x,\widehat y)$ of ${\rm D} f (\widehat x,\widehat y)$ satisfy
\begin{equation}\label{eq:pc}
P(\mu )= \mu ^2 -T(\widehat x,\widehat y)\mu +D(\widehat
x,\widehat y)=\mu ^2 -T \mu +D=0, \end{equation}
 where, when
there is no confusion, we omit the dependence of $T,D$ and $\mu $
on the fixed point.

With this notation, the condition of having complex eigenvalues
$\mu $ and $\bar\mu$ with modulus different from one
reads simply as $T^2-4D<0$ and $D\ne1,$ because they satisfy \eqref{eq:pc},
$\mu \bar\mu =D,$ and hence $|\mu |^2=D^2.$

By Theorem \ref{t1}, if $f$ has a meromorphic first integral, then there
exists $(0,0)\ne(p,q)\in\Z^2,$ such that $\mu ^p\bar\mu ^q=1.$
Taking norms, this means that $|\mu |^{p+q}=1$. Since by
hypothesis $|\mu |\neq1,$  we must have $q=-p$. Hence,
$|{\mu }/{\bar \mu }|=1$ and  $\mu /\bar\mu $ is a
$p$-root of the unity. Therefore
$\operatorname{Re}(\mu /\bar\mu )=\cos (2n\pi/p)$ for some
$n\in\{0,1,\ldots,p-1\}$ with $n$ and $p$ coprime.

Let us write
\[
\mu =\frac T2+{\rm i}\sqrt{\frac{4D-T^2}4}=\alpha+{\rm i}\beta,
\quad \bar\mu =\frac T2-{\rm i}\sqrt{\frac{4D-T^2}4}=\alpha-{\rm
i}\beta.
\]
Then
\[
\frac{\mu }{\bar\mu }=\frac{\alpha+i\beta}{\alpha-i\beta}=\frac{\alpha^2-\beta^2}{\alpha^2+\beta^2}+{\rm
i}\frac{2\alpha\beta}{\alpha^2+\beta^2}
\]
and, as a consequence,
\[
\operatorname{Re}\left(\frac{\alpha^2-\beta^2}{\alpha^2+\beta^2}\right)=\frac{T^2}{2D}-1=\cos\left({2n\pi}/p\right).
\]
If we name $v=\cos\left({2n\pi}/p\right),$ then the
following system of three equations is satisfied:
\begin{equation}\label{eq:3eq}
\begin{cases}
&S_1(x,y)=0,\quad S_2(x,y)=0,\\
& W(x,y,v)=T_1^2(x,y)D_2(x,y)-2(1+v)T_2^2(x,y)D_1(x,y)=0.
\end{cases}
\end{equation}
As usual, taking successive resultants, as follows
\begin{align*}
T_1(y)&=\operatorname{Res}_x\big(S_1(x,y),S_2(x,y)\big),\quad T_2(y,v)=\operatorname{Res}_x\big(S_1(x,y),W(x,y,v)\big),\\
U_k(v)&=\operatorname{Res}_y\big(T_1(y),T_2(y,v)\big),
\end{align*}
where $k$ is de degree of $U_k,$ we know that $U_k(v)=0$ and
$U_k\in\Q[x],$ as we wanted to prove, see again~\cite{stu}. Hence,
by Proposition \ref{pr:trig},  $v$ has to be a root of some $
V_{p},$ with $ \deg(V_p(x))\le k.$   By the properties of the
resultant the last statement of the theorem follows.
\end{proof}

\begin{remark}\label{rem} Notice that the conditions in Theorem \ref{te:t0}
depend on the value $k$ given in its statement and obtained solving
system \eqref{eq:3eq}. According to the fixed point $(\widehat x,\widehat
y)$ taken into account, some smaller $k$ can be  considered.
 To do this, let
us treat this system in another way. We consider first the following
two polynomials with rational coefficients,
\begin{align*}
T_1(y)&=\operatorname{Res}_x\big(S_1(x,y),S_2(x,y)\big),\quad
T_3(x)=\operatorname{Res}_y\big(S_1(x,y),S_2(x,y)\big).
\end{align*}
Let $V_1$ and $V_3$ be the irreducible factors in $\Q[x]$ of $T_1$ and
$T_3,$ respectively, such that $T_1(\widehat y)=0$ and $T_3(\widehat
x)=0.$ Then we can follow the same procedure that for system
\eqref{eq:3eq}, but starting with the system
\begin{equation}\label{eq:red}
 V_1(y)=0,\quad V_3(x)=0,\quad  W(x,y,v)=0.
\end{equation}
We arrive at a new $U_{k'}\in\Q[x],$ of degree $k'\le k$ and
such that $U_{k'}(v)=0.$ Then the set of conditions \eqref{eq:con}
has to be satisfied only until $k',$ taking $U_{k'}$ instead of
$U_k.$
\end{remark}

\begin{corollary}\label{co:deg1}
Consider the rational map \eqref{ratmap1} with rational
coefficients. Assume that $f$ has a fixed point $(\widehat x,\widehat
y)\in\Q^2$ and ${\rm D} f (\widehat x,\widehat y)$ has complex
conjugated eigenvalues, $\mu ,\bar\mu ,$ solution of
\[
\mu ^2-T(\widehat x,\widehat y)\mu +D(\widehat x,\widehat
y)=0,
\]
and with modulus different of 1. If the map \eqref{ratmap1}
has a meromorphic first integral, then
\[
\frac{T^2(\widehat x,\widehat y)}{D(\widehat x,\widehat
y)}\in\left\{0,1,2,3,4\right\}.\]
\end{corollary}

\begin{proof} By using Theorem \ref{te:t0}, Remark \ref{rem} and that $(\widehat x,\widehat y)\in\Q^2$, we
know that system \eqref{eq:red} writes as
\[
V_1(y)=y-\widehat y,\quad  V_3(x)=x-\widehat x,\quad W(x,y,v)=0.
\]
Hence,
\[
v=\frac{T^2(\widehat x,\widehat y)}{2D(\widehat x,\widehat
y)}-1\in\Q
\]
and $k=1$ in Theorem \ref{te:t0}. Therefore, the only possible
values of $v=\cos(2n\pi/p)$ are the ones that are rational. From
\eqref{eq:rat} we get that
\[\frac{T^2(\widehat x,\widehat
y)}{2D(\widehat x,\widehat y)}-1 \in\left\{-1,-\frac 12,0,\frac
12,1\right\}\] and the corollary follows.
\end{proof}

In the next lemma we apply the above corollary to a simple example.

\begin{lemma}\label{le:ex1}
The rational map
\begin{equation*}%\label{Eq.Frat}
f(x,y)=\left(x y, \frac{(a+(2-a) x) y}{1 + x y}\right),
\end{equation*}
with $a\in\Q,a>9/8$, can have   meromorhic first integrals only when
$a\in\{3/2,2,9/4,9/2\}.$
\end{lemma}

\begin{proof}
It has a fixed point at $(1,1)$; the characteristic polynomial of
${\rm D} f (1,1)$ is $p(\mu )=\mu ^2-3/2\mu +a/2$ and their
eigenvalues  $(3\pm {\rm i} \sqrt{8a-9})/4$ are complex. Moreover,
when $a\ne2$ they have modulus different from 1. Notice that
${T^2(\widehat x,\widehat y)}/{D(\widehat x,\widehat y)}=9/(2a).$
Hence, by Corollary \ref{te:t0}, when $a\ne2,$ the conditions for
existence of meromorphic first integral are
\[
\frac9{2a}\in\{0,1,2,3,4\},
\]
giving the result of the statement.
\end{proof}

Although we have not found any meromorphic first integral for the
four possible cases given in the statement it is worth to mention
that when $a=1$ then $f$ has the meromorphic (polynomial) first
integral $H(x,y)=y(1+x),$ see also \cite{SS}. In fact, let us see
which conditions for existence of meromorphic first integral are
consequence of Theorem \ref{t1} when  $a<9/8.$ In this case, let
$r\in\R$ be the non-negative solution of $r^2=9-8a$. Then the
eigenvalues at $(\widehat x,\widehat y)$ are $\mu _1=(3-r)/4$ and
$\mu _2=(3+r)/4$. Hence, the condition for the existence of a
meromorphic first integral is that $\mu _1^n\mu _2^m=1,$ for
some non-zero $(n,m)\in\Z^2.$ Equivalently, the only cases that
might have a meromorphic first integral  are either $r=1$ or
\[
{\log\Big(\frac{3-r}4\Big)}\Big/{\log\Big(\frac{3+r}4\Big)}\in\Q.
\]
Notice that $r=1$ precisely corresponds to the integrable case given
above $a=1.$

\bigskip

Next result is similar to Corollary \ref{co:deg1}.  The only
difference is that each one of the coordinates of the fixed point is a
zero  of a quadratic polynomial with integer coefficients. Its proof is
essentially the same, but instead of using  \eqref{eq:rat} we will
use \eqref{eq:sqr} and, consequently, that the only values of
$\cos(2n\pi/p),$ that are solutions of a  quadratic polynomial in
$\Q[x],$ irreducible in $\Q[x],$ are
\begin{equation*}
\frac{-1\pm\sqrt 5}4,\pm\frac{\sqrt 2}2,\frac{1\pm\sqrt
5}4,\pm\frac{\sqrt 3}2.
\end{equation*}
We skip the details of the proof.

\begin{corollary}\label{co:deg2}
Consider the rational map \eqref{ratmap1} with rational
coefficients. Assume that $f$ has a fixed point $(\widehat x,\widehat
y)$ with both coordinates in $\Q[\sqrt s]^2\setminus\Q^2$, $s\in\Q$,
and ${\rm D} f (\widehat x,\widehat y)$ has complex conjugated
eigenvalues $\mu ,\bar\mu$ solution of
\[
\mu ^2-T(\widehat x,\widehat y)\mu +D(\widehat x,\widehat
y)=0,
\]
and with modulus different from 1. If the map \eqref{ratmap1}
has a meromorphic first integral, then
\[
\frac{T^2(\widehat x,\widehat y)}{D(\widehat x,\widehat
y)}\in\left\{\frac{3 \pm \sqrt 5}2,  2 \pm \sqrt 2,\frac{5 \pm \sqrt
5}2, 2 \pm\sqrt 3\right\}.\]
\end{corollary}

\bigskip

Similar results to Corollaries \ref{co:deg1} and \ref{co:deg2} could
be stated assuming that the rational map  $f$ given in
\eqref{ratmap1} has a fixed point $(\widehat x,\widehat y)$ with both
coordinates being a zero of a polynomial of higher degree  and the
same conditions for the eigenvalues of ${\rm D} f (\widehat
x,\widehat y)$ hold. As an example we apply our techniques to prove
the non-meromorphic integrability of a concrete rational map.
Consider
\begin{equation}\label{eq:conc}
f(x,y)=\left(x+y^2-xy,\frac{x^2+xy+1}{x^2-3y+1}\right),
\end{equation}
and its fixed point $(\widehat x,\widehat y)=(s,s)$, where $s\approx 4.836$ is the real root of
$P(x)=x^3-5x^2+x-1.$

It is easy to see that we are under the hypotheses of Theorem
\ref{te:t0}, because $T^2(s,s)-4D(s,s)<0$ and $D(s,s)\ne1.$ For the
sake of shortness, we omit the explicit expressions of $T(x,y)$ and
$D(x,y).$ Moreover, in the notation of this theorem,
\begin{equation*}
 S_1(x,y)=y(y-x),\quad S_2(x,y)=-x^2y+x^2+xy+3y^2-y+1
\end{equation*}
and $W(x,y,v)$ is a polynomial of degree 10 that we do not explicit
either. Then,
\begin{align*}
T_1(y)&=\operatorname{Res}_x\big(S_1(x,y),S_2(x,y)\big)=-y^2P(y),\\
T_3(x)&=\operatorname{Res}_y\big(S_1(x,y),S_2(x,y)\big)=-(x^2+1)P(x).
\end{align*}
Therefore, in the notation of Remark \ref{rem}, $V_1(y)=P(y)$ and
$V_3(x)=P(x).$ Finally,
\[
\operatorname{Res}_x\big(\operatorname{Res}_y(W(x,y,v),V_1(y)),V_3(x)\big)
=U_3(v)Z_6(v),
\]
where $U_3,Z_6\in\Q(x),$  $U_3(x)=5833x^3+16607x^2+15650x+4874$ and
$U_3(\operatorname{Re}(\mu /\bar\mu ))=0,$ with $\mu $ and $\bar\mu
$ the eigenvalues of ${\rm D} f (s,s).$ Since $U_3$ has degree $3,$
to prove that system \eqref{eq:conc} has no meromorphic first
integral it suffices to prove that all the resultants $\operatorname
{Res}_x(U_3(x),Q(x)), Q\in\mathcal{Q},$ do not vanish, where
\begin{align*}
\mathcal{Q}=\{&x,\,x-1,\, x+1,\,2x+1,\, 2x-1,\, 2x^2-1,\, 4x^2-3,\,
4x^2+2x-1,\, 4x^2-2x-1,\\ &8x^3+4x^2-4x-1,\,8x^3-6x+1,\,
8x^3-4x^2-4x+1,\,8x^3-6x-1\}.
\end{align*}
The above set of polynomials corresponds to the only irreducible
ones of degree at most 3 that have a root $\cos(2n\pi/p)$ for some
$n,p\in\Z,$ see the comments after Proposition \ref{pr:trig}. These
thirteen resultants are all different from zero, and the result
follows.

\subsection{Higher dimensional examples}

As in the two dimensional case,  the study of the sequences
generated by $n$-th order difference equations
$$x_{k+n}=g(x_k,x_{k+1},\ldots,x_{k+n-2},x_{k+n-1})$$ can be reduced
to the study of the DDS generated by the  map \[
f(x_1,x_2,\ldots,x_n)=(x_2,x_3,\ldots,x_{n-1},g(x_1,x_2,\ldots,x_n)),\]
from $\C^n$ into itself. Some of the maps that we will consider here have this shape.

\smallskip

We start studying the case of general analytic maps having the maximum possible
number of functionally independent meromorphic first integrals: $n.$

\begin{proposition}\label{pr:rn} Let $f$ be an analytic map from $\C^n$ into itself, with isolated fixed points,
and with $n$ functionally independent meromorphic first integrals.
If ${\bf q}$ is a fixed point of $f,$ then all the eigenvalues~$\mu ,$
of ${\rm D} f({\bf q})$ are $p$-roots of the unity. Moreover, there
is a constructive procedure, described in the proof,
to find a polynomial $P_k\in\C[x]$ of degree $k$ such that, for all
$\mu ,$ $P_k(\mu )=0.$

Additionally, if $f$ is a rational map and all the numerators and
denominators of its components are polynomials in
$\Q[x_1,\ldots,x_n]$, then $P_k\in\Q[x]$ and  $p\le M,$ where $M$ is
the maximum $m\in\N$ such that the degree of the cyclotomic
polynomial $\Phi_m$ is $k.$ Furthermore,
$\operatorname{Res}_x(\Phi_j(x),P_k(x))=0$ for some  cyclotomic
polynomial $\Phi_j$ with $\deg(\Phi_j)\le k.$
\end{proposition}

\begin{proof}
    By Theorem \ref{t1}, the dimension of the
    $\mathbb Z$--linear space generated from $\{\mathbf k\in\mathbb
    Z^n:\, \mu^{\mathbf k}=1\}$ is $n.$ Let ${\mathbf k}_i, i=1,2,\ldots,n$ be a basis of this space.
     Then, $\det(K)\ne0,$ where $K$ is the $n\times n$ matrix $K=({\mathbf k}_{i,j})$ and, moreover
     it can be seen that $\mu_\ell^{|\det (K)|}=1,$ for all $\ell=1,2,\ldots,n,$ proving that $p=|\det (K)|\in\N.$

    Set ${\mathbf x}=(x_1,x_2,\ldots,x_n)$. Then, the eigenvalues of ${\rm D} f$ at a fixed point are the
    solutions ${\mathbf x}, \mu$ of the system of $n+1$ equations,
     $f({\mathbf x})={\mathbf x},\, P({\mathbf x},\mu)=0,$ where $ P({\mathbf x},\mu)$
     is the characteristic polynomial of ${\rm D} f({\mathbf x})$ at an arbitrary point ${\mathbf x},$ not necessarily fixed.

\smallskip

Instead of considering them, we take their numerators as a
polynomial system of $n+1$ equations and $n+1$ unknowns,
\begin{equation}\label{eq:sysn}
g({\mathbf x})=\operatorname{Num}(f({\mathbf x})-{\mathbf
x})=0,\quad R({\mathbf x},\mu)=\operatorname{Num}(P({\mathbf
x},\mu))=0,
\end{equation}
where $\operatorname{Num}$ denotes the numerator of a quotient of
polynomials. Doing successive resultants following \cite{stu}, as in
the proof of Theorem \ref{te:t0}, and because $f$ has no continuum
of fixed points, we arrive at a non-zero polynomial $P_k\in\C[x]$,
such that $P_k(\mu)=0$ for all eigenvalues $\mu$ of $Df({\mathbf
x})$ at any ${\mathbf{x}},$ fixed point of $f,$ as we wanted to
prove. Moreover, if $f$ is rational, with numerator and denominator
in $\Q[{\mathbf x}],$ then $P_k\in\Q[x].$

Finally, it is known that given a primitive  $p$-th root of the
unity, $\mu,$ the minimal degree of a polynomial  $S\in\Q[x]$ such
that $S(\mu)=0$ is $\phi(p)=\deg(\Phi_p(x)).$
 Hence, $p\le M,$ as we wanted to prove. Moreover $P_k$ must share some root with one of
 the polynomials $\Phi_j$ with $\phi(j)\le k,$ and as a consequence, for this value of $j,$
 $\operatorname{Res}_x(\Phi_j(x),P_k(x))=0.$
\end{proof}

As an example of application we prove the following lemma.

\begin{lemma} Consider the map
    \begin{equation}\label{eq:todd}
    f(x,y,z)=\left(y,z,\frac{a+y+z}{x}\right),\quad a\in\Q.
    \end{equation}
    If it has 3 functionally independent meromorphic first integrals, then $a\in\{-1,1\}.$
\end{lemma}

\begin{proof}
    We will apply Proposition \ref{pr:rn}. We start constructing a reduced, but equivalent, version of system \eqref{eq:sysn}.
     Notice that the
    fixed points of $f$ are $(x,x,x)$ such that $a+2x=x^2$ and
    \[
{\rm D} f (x,x,x)= \left(
           \begin{array}{ccc}
             0 & 1 & 0\\
             0&0&1\\
               -\frac{a+y+z}{x^2} &\frac 1 x & \frac 1 x\\
           \end{array}
         \right)_{(x,y,z)=(x,x,x)} = \left(
           \begin{array}{ccc}
             0 & 1 & 0\\
             0&0&1\\
               -1 &\frac 1 x & \frac 1 x\\
           \end{array}
         \right).
\]
Hence, the characteristic polynomial at any fixed point is
$(\mu+1)(\mu^2-(1+1/x)\mu+1).$ Therefore, since $\mu=-1$ is a root
of the unity,  system \eqref{eq:sysn} with 4 unknowns can be simply
reduced to a system of 2 equations and 2 unknowns, $x$ and $\mu,$
and a parameter $a$:
\[
-x^2+2x+a=0,\quad (\mu^2-\mu+1)x-\mu=0.
\]
Hence
\[
P_4(\mu)=\operatorname{Res}_x(-x^2+2x+a,(\mu^2-\mu+1)x-\mu)=
a\mu^4-2(a-1)\mu^3+3(a-1)\mu^2-2(a-1)\mu+a=0.
\]
Therefore, we only have to consider the values of $p$ for which  the
cyclotomic polynomial $\Phi_p$ has degree at most 4. They are
$p\in\mathcal{P}=\{1,2,3,4,5,6,8,10,12\}.$ By doing all the
resultants $\operatorname{Res}_x(P_4,\Phi_p(x)),$ with
$p\in\mathcal{P},$ we can discard all the values  $a\in\Q$ but
$a=-1,7/9,5/4,3,1,$ that correspond to $p=1,2,3,4,8,$ respectively.
For instance, we get that
$\operatorname{Res}_x(P_4,\Phi_3(x))=(4a-5)^2$ and
$\operatorname{Res}_x(P_4,\Phi_8(x))=(a-1)^2,$ but
$\operatorname{Res}_x(P_4,\Phi_{10}(x))=(a^2-a-1)^2.$ Finally, the
values  $a\in\{7/9,5/4,3\}$ are also discarded  because, for them,
some of the roots of $P_4$ are not roots of the unity. For instance,
$\left.P_4(\mu)\right|_{a=3}=(\mu^2+\mu+1)(5\mu^2-7\mu+5)/4.$ On the
other hand,
\[
\left.P_4(\mu)\right|_{a=-1}=-(\mu-1)^4 \quad \mbox{and}\quad
\left.P_4(\mu)\right|_{a=1}=\mu^4+1,
\]
and all the roots of both polynomials are roots of the unity.
\end{proof}

It is worth to comment that for $a=1$ it is well known that the map
\eqref{eq:todd} has effectively~3 functionally independent rational
first integrals. Two of them exist for any $a\in\C.$ They are
\begin{align*}
H_1(x,y,z)&=\frac{(x+1)(y+1)(z+1)(a+x+y+z)}{xyz},\\
H_2(x,y,z)&=\frac{(1+x+y)(1+y+z)(a+x+y+z+xz)}{xyz},
\end{align*}
see for instance \cite{CGM5} and its references, and a third one can
be seen in \cite{CGM}  and it is found by using  the tools
introduced in that paper. It exists because this map, for $a=1,$
corresponds to the celebrated 3$^{\rm rd}$ order Todd's difference
equation $x_{n+3}=(1+x_{n+2}+x_n)/x_n$ which is globally 8-periodic,
that is $x_{n+8}=x_n$ for all $n,$ whenever $x_k$ is well defined.
We believe that when $a=-1$ the map has only the above 2
functionally independent meromorphic first integrals, but from our
approach we cannot discard the existence of a third one. In fact, it
is known, even for $a\in\C,$ that the only globally periodic map
corresponds to $a=1,$ see \cite{CGM4,CL}.

\bigskip

There is also a simple case for which the non-existence of
meromorphic first integral can be easily established.

\begin{corollary} Let $f$ be an  analytic diffeomorphism
with a fixed point ${\bf q}\in\C^n$ and assume that the eigenvalues
of ${\rm D} f ({\bf q})$ are $n$ different prime numbers. Then $f$
has no meromorphic first integral.
\end{corollary}

\begin{proof} By Theorem \ref{t1} we have to calculate the dimension of the
$\mathbb Z$--linear space generated from $\{\mathbf k\in\mathbb
Z^n:\, \mu^{\mathbf k}=1\}$ where here $\mu=(p_1,p_2,\ldots,p_n)$
and the $p_j$ are the $n$ different prime numbers. Hence, from the
condition $p_1^{k_1}p_2^{k_2}\cdots p_n^{k_n}=1$ we get
$k_1=k_2=\cdots=k_n=0$ and this dimension is $0.$ As a consequence,
the map has no meromorphic first integral.
\end{proof}

\bigskip

There are many other rational maps in dimension $n>2$ that admit
meromorphic (indeed rational) first integrals. As a final example we
show one of them, obtained from the paper \cite{sch}, dedicated to
study systems of difference equations with invariants.

\smallskip

The map
\[
f(x,y,z,t)=\left(z,t,{\frac {az+bt+c}{x}},{\frac
{az+bt+c}{y}}\right)
\]
has the first integral
\[
H(x,y,z,t)=\frac{\left( xy+ay+bx \right)
 \left( zt+at+bz \right)  \left(
 ax+az+bt+by+c
 \right)}{xyzt}.
\]
By using our result we could embed  the above map into a large
family, with more parameters, and find necessary conditions among
them for the  cases with one or more meromorphic first integrals.

\subsection{Acknowledgements} This work has received funding from the Ministerio de Ciencia e Innovaci\'{o}n
(MTM2016-77278-P FEDER and PID2019-104658GB-I00 grants) and the
Ag\`{e}ncia de Gesti\'{o} d'Ajuts Universitaris i de Recerca (2017 SGR 1617
grant). The third author is partially supported by NNSF of China
grant numbers 11671254, 11871334 and 12071284, and also by
Innovation Program of Shanghai Municipal Education Commission.

\end{document}